\documentclass[12pt]{amsart}
\usepackage[margin=0.8in]{geometry}
\usepackage{tikz}
\usepackage{amsmath}
\usepackage{amsfonts}
\usepackage{amsthm}
\usepackage{amssymb}
\usepackage{times} 
\usepackage{graphicx}
\numberwithin{equation}{section}
\usetikzlibrary{arrows}
\usepackage{enumerate}
\newtheorem{thm}{Theorem}[section]

\newtheorem{con}[thm]{Conjecture}
\newtheorem{lem}[thm]{Lemma}
\newtheorem{prop}[thm]{Proposition}
\newtheorem{rmk}[thm]{Remark}
\newtheorem*{defn}{Definition}
\newtheorem{example}[thm]{Example}
\theoremstyle{definition}
\newtheoremstyle{named}{}{}{\itshape}{}{\bfseries}{.}{.5em}{\thmnote{#3 }#1}
\theoremstyle{named}
\newtheorem*{namedthm}{Theorem}

\tikzset{node distance=5cm, auto}
\date{}
\begin{document}
\title{A finiteness theorem for  special unitary
groups of quaternionic skew-hermitian forms with good reduction}
\author{Srimathy Srinivasan} 
\address{School of Mathematics, Institute for Advanced Study, Princeton NJ, USA - 08540}
\email{srimathy@ias.edu}
\thanks{This material is based upon work supported by the National Science Foundation under Grant No. DMS - 1638352.}

\date{}

\begin{abstract}
Given a field $K$ equipped with a set of discrete valuations $V$, we develop a general theory to relate reduction properties of  skew-hermitian forms over a quaternion $K$-algebra $Q$  to   quadratic forms over the function field $K(Q)$ obtained via Morita equivalence.  Using this we show that if  $(K,V)$ satisfies certain conditions, then the number of $K$-isomorphism classes of the universal coverings of the special unitary groups of quaternionic skew-hermitian forms that  have good reduction at all valuations in $V$ is finite and bounded by a value that depends on  size of a quotient of the Picard group of $V$ and the size of the  kernel and cokernel of  residue maps in Galois cohomology of $K$ with finite coefficients.   As a corollary we prove  a conjecture of Chernousov, Rapinchuk, Rapinchuk for groups of this type.
\end{abstract}
\maketitle
\section{Introduction}\label{sec:intro}
The concept of good reduction of elliptic curves is well studied in the literature and can be characterized by unramified points of finite order. It is also well understood in the more general setting of abelian varieties (\cite{serre}, \cite{faltings}). In the case of linear algebraic groups, the study of good reduction basically started with Harder (\cite{harder}) with focus on number fields. This was followed by more progress in this direction due to many authors. It was not until very recently, that this topic was approached in a more general setting by Chernousov, Rapinchuk, Rapinchuk \cite{rapinchuk_spinor} where they consider arbitrary finitely generated fields and provide results for two dimensional global fields as well as for  some fields that are not finitely generated.\\
\indent Given a discrete valuation $v$ on a field $K$,  let $K_v$, $\mathcal{O}_v$,  and $K^{(v)}$ denote respectively the the completion of $K$, its valuation ring  and the residue field. Assume that all the fields under consideration have characteristic $\neq 2$.   Let $G$ be an absolutely almost simple linear algebraic group defined over $K$.  Then $G$   is said to have \emph{good reduction at $v$}  if there exists a reductive group scheme  $\mathcal{G}$ over $\mathcal{O}_v$ (see Definition 2.7, Exp. XIX, \cite{sga3} for the definition of reductive group schemes) with generic fiber $\mathcal{G} \otimes_{\mathcal{O}_v} K_v$  isomorphic to $G \otimes_K K_v$.  In a recent paper, Chernousov, Rapinchuk, Rapinchuk (\cite{rapinchuk_spinor}) asked the following question:\\
\indent \emph{``Given an absolutely almost simple simply connected algebraic $K$-group $G$, can one equip $K$ with a set of discrete valuations $V$ such that the set of $K$-isomorphism classes of $K$-forms of $G$ having good reduction at all $v \in V$ is finite?"}\\
\indent An affirmative answer to the question has important implications such as  properness of local-global map in Galois cohomology, finiteness of genus of an algebraic group (see \cite{crr_genus}, \cite{rapinchuk_spinor} for the definition of genus and its relation to good reduction) and in eigenvalue rigidity problems. A detailed explanation  of why this question is  important and interesting  can be  found in   \cite{crr_finite} and \cite{rapinchuk_spinor}. \\
\indent  It is well known that the answer to the question is affirmative  over   number fields  where $V$ can be chosen to be the set containing almost all non-archimedean places(\cite{gross_z}, \cite{conrad_z}, \cite{javan_z}). In the general setting of arbitrary fields, when studying good reduction of classical algebraic groups one is inevitably led to analyzing the underlying sesquilinear form that defines the  group. In particular,  reduction properties of the special unitary groups $SU(h)$ where  $h$ is a non-degenerate hermitian/skew-hermitian form over a division algebra with involution (Types A, C, D) and the spinor groups $Spin(q)$ of non-degenerate quadratic forms over $K$ (Types B, D) is  related to reduction properties of the underlying  forms $h$ and $q$ respectively. Thus the above question on  finiteness  of number of  isomorphism classes of $K$-forms of $G$  with good reduction at a set of valuations is reduced to  asking if  the number of similarity classes of forms associated to the type of $G$ that have good reduction at these valuations is finite.  In the case of spinor groups $Spin(q)$, since the underlying form is quadratic,  one can use Milnor isomorphism to map its Witt class to Galois cohomology class and take advantage of powerful cohomological tools to prove such finiteness theorems. This is the methodology adopted in \cite{rapinchuk_spinor} to prove finiteness results  for the class of  spinor groups $Spin(q)$  defined  over  $K = k(C)$,  function field of a smooth geometrically integral curve $C$ over a field $k$ of characteristic $\neq 2$ that satisfies condition ($\text{F}_2'$) (See \S\ref{sec:fm} for the definition, properties and examples of fields of type ($\text{F}_m'$), examples include local fields and higher dimensional local fields)  and $V$ is the set of valuations corresponding to closed points of $C$. In the case of special unitary groups $SU(h)$ where  underlying form $h$ is hermitian over a quadratic extension of $K$ or hermitian over a quaternion division algebra with center $K$, one applies Jacobson's theorem (\cite{jacobson}) to associate  $h$ to  a quadratic form $q_h$ of higher rank over $K$ so that   reduction properties  can be be studied via  ramification of  $q_h$ using  cohomological methods  as before.  This yields required finiteness results for the  special unitary groups $SU(h)$ for $h$ as above (see \S8 in \cite{rapinchuk_spinor}). Based on the above evidence  Chernousov, Rapinchuk, Rapinchuk  conjectured that with mild assumptions such finiteness results must hold for any  absolutely almost simple simply connected algebraic $K$-group. We restate the conjecture below:
 
\begin{con} \label{con:conjecture}(Conjecture 7.3  in \cite{rapinchuk_spinor}) : Let $K= k(C)$ be the function field of a smooth affine geometrically integral curve over a field $k$ and let $V$ be the set of discrete valuations associated with the closed points  of  $C$. Furthermore, let $G$ be an absolutely almost simple simply connected algebraic $K$-group and let $m$ be the order of the automorphism group of its root system. Assume that $char~k$ is prime to $m$ and that $k$ satisfies ($F_m'$). Then the set of $K$-isomorphism classes of $K$-forms of $G$ that have good reduction at all $v \in V$ is finite. 
\end{con}

An important missing link to prove the conjecture is the case of the universal covering of  the special unitary groups $SU(h)$ of skew-hermitian form $h$ over quaternion $Q$ over $K$. See Remark 8.6 in \cite{rapinchuk_spinor}.  In this paper, we provide answer to this missing link:
\begin{namedthm}[Main]\label{thm:main1}
Let $K= k(C)$ be the function field of a geometrically integral curve over a field $k$ of characteristic $\neq 2$ that satisfies ($\text{F}_2'$) and let $V$ be the  set of discrete valuations on $K$ corresponding to closed points of $C$.  Then the number of $K$-isomorphism classes of the universal coverings of special unitary groups $SU_m(h)$ of non-degenerate  $m$-dimensional skew-hermitian forms $h$  over some  quaternion $K$-algebra with the canonical  involution, having good reduction at all places in $V$ is finite.
\end{namedthm}
This is obtained as a consequence of a general theory that we briefly outline below without getting into  technical details. Let $h$ be  skew-hermitian form over a quaternion division algebra $Q$ with center $K$ and let $K(Q)$ denote the function field of the Severi-Brauer variety associated to $Q$. Then the form $h_{K(Q)}:= h \otimes_{K} K(Q)$ can be reduced to a quadratic form $q_{h_{K(Q)}}$ via Morita equivalence.  We give a  method to extend a given valuation $v$ on $K$ to a valuation $\tilde{v}$ on $K(Q)$ in such a way that if $h$ has good reduction at $v$ then the quadratic form $q_{h_{K(Q)}}$ is unramified at $\tilde{v}$ (see \S\ref{sec:unramified} for the notion of ramification and good reduction of forms). We then use cohomological methods to show finiteness results for certain unramified Galois  cohomology classes over $K(Q)$.  This allows us to derive an upper bound on the number of $K$-isomorphism classes of special unitary groups  of quaternionic skew-hermitian forms with good reduction at all  places in $V$ thereby proving Conjecture \ref{con:conjecture} for groups of this type.

\section{Notations} All the fields we consider here have characteristic $\neq 2$.  For a discrete valuation $v$ on  a field $F$, let $\mathcal{O}_v$ (or  $\mathcal{O}_F$) denote the valuation ring in $F$ when the underlying field $F$ (resp. the underlying valuation) is clear.  Let $F_v$ denote  the completion of $F$ with respect to $v$ and let $F^{(v)}$ denote its residue field. 
 The Brauer group of $F$ is denoted by $Br(F)$ and its $n$-torsion subgroup by ${}_{n}Br(K)$. For a  quaternion  algebra $Q$ over $F$, we denote  its  Brauer class by $[Q]$. The involution on $Q$ is the canonical (symplectic) involution.  Let $Q_v := Q\otimes_F F_v$. For a ring $R$, let $R^*$ denote its group of units. All the forms considered  in this paper are finite dimensional and non-degenerate. The Witt ring of $F$ is denoted by $W(F)$ and $\mathcal{I}(F)$ denotes its fundamental ideal. For a quadratic form $q$ over $F$, let $[q]$ denote its class in the Witt ring $W(F)$.

\section{Outline of the proof of the Main Theorem} \label{sec:main}
Let $K$ be a field equipped with a discrete valuation $v$ where $char~ K^{(v)} \neq 2$. Let $Q$ be a quaternion division algebra over $K$.   Assume that  $Q$ is unramified at  $v$ (see \S\ref{sec:unramified} for the notion of unramified quaternion).  Let $K(Q)$ denote the function field of the Severi-Brauer variety associated to $Q$. Let $h$ be a non-degenerate skew-hermitian form over $Q$ and let $q_{h_{K(Q)}}$ be the quadratic form associated to $h_{K(Q)}:= h \otimes_{K} K(Q)$ via Morita equivalence (see \S\ref{sec:morita}). We refer the reader to \S\ref{sec:unramified} for the notion of ramification and good reduction of forms.
\begin{thm}\label{thm:cool}
There exists a valuation $\tilde{v}$ on $K(Q)$ extending $v$ such that  if $h$ has good reduction at $v$, then $q_{h_{K(Q)}}$ is unramified at $\tilde{v}$.
\end{thm}
\begin{proof}
See \S\ref{sec:ext}.
\end{proof}

We  use the above result to prove finiteness statements as claimed. 
Recall the following notations and  facts from \cite{rapinchuk_spinor}. Let $V$ be a set of discrete valuations on  a field $K$ that satisfies the following conditions.
\begin{enumerate}[(A)]
\item For any  $a \in K^*$, the set $V(a): = \{v \in V ~|~ v(a) \neq 0 \}$ is finite.
\item $char ~ K^{(v)} \neq 2~\forall v \in V$.
\end{enumerate}
As noted in \cite{rapinchuk_spinor}, condition (A) is satisfied for a divisoral set of valuations $V$ on a finitely generated $K$.   Due to conditions (A) and (B),  for $l$ a power of $2$, we  have residue maps in Galois cohomology (See \S\ref{sec:unramified} for a discussion on this)
\begin{align} \label{eqn:res}
r_l^i:  H^i(K, \mu_l^{\otimes i-1}) &\rightarrow \bigoplus_{v \in V} H^{i-1}(K^{(v)}, \mu_l ^{\otimes i-2})
\end{align}
where $\mu_l$ denotes $l$-th roots of unity.
For $l=2$, this is just 
\begin{align}\label{eqn:res2}
r_2^i: H^i (K , \mu_2) &\rightarrow \bigoplus_{v \in V} H^{i-1} (K^{(v)},  \mu_2)
\end{align}
Let $H^i(K, \mu_l^{\otimes i-1})_{V}$ and $H^i(K, \mu_l^{\otimes i-1})^{V}$ denote  respectively the $Ker ~r_l^i$ and  $Coker ~r_l^i$.  Let $Pic(V)$  denote the Picard group of $V$ (see \S2 in \cite{rapinchuk_spinor} for the definition).
The main theorem is a consequence of  the following.
\begin{thm}\label{thm:main}
Let $V$ be a set of discrete valuations on $K$  satisfying conditions (A) and (B). Suppose the following holds:
\begin{enumerate}
\item the quotient $Pic(V)/2 Pic (V)$ is finite and 
\item the cohomology groups  $H^2 (K , \mu_2)_{V}$, $H^i (K , \mu_2)^{V}$ and $H^i(K, \mu_4^{\otimes i-1})_{V}$ are finite for all $i = 1,2, \cdots l : = [log_2~ 2n] +1$.
\end{enumerate}
Then the number of $K$-isomorphism classes of special unitary groups of $n$-dimensional skew-hermitian forms  $SU_n(h)$, where $h$ is  skew-hermitian over some quaternion (not necessarily division) algebra with the canonical  involution, having good reduction at all $v \in V$ is finite and bounded above by 
\begin{align}\label{eqn:bound}
|H^2 (K , \mu_2)_{V}| \cdot |Pic(V)/2 Pic(V)| \cdot  \prod_{i=1}^{l} |H^i (K , \mu_2)^{V}|\cdot |H^i(K, \mu_4^{\otimes i-1})_{V}|
\end{align} 
\end{thm}
\begin{proof}
See \S\ref{sec:proofmain}.
\end{proof}

A situation where the hypothesis of Theorem \ref{thm:main} holds is the following. 
\begin{prop} \label{prop:hypothesis}
Let $K= k(C)$ be the function field of a geometrically integral curve $C$ over a field $k$ of characteristic $\neq 2$ that satisfies ($\text{F}_2'$) (see  \S\ref{sec:fm})  and let $V$ be the  set of discrete valuations on $K$ corresponding to closed points of $C$.  Then $(K, V)$ satisfies the hypothesis of Theorem \ref{thm:main}.
\end{prop}
\begin{proof}
Conditions (A) and (B) are easily seen to be satisfied. The finiteness of $Pic(V)/2 Pic(V)$ is shown in the proof of Theorem 1.4 in \cite{rapinchuk_spinor}. We will now show finiteness of $H^2 (K , \mu_2)_{V}$, $H^i (K , \mu_2)^{V}$ and $H^i(K, \mu_4^{\otimes i-1})_{V}$. For any $l$ coprime to $char~K^{(v)}$, the Bloch-Ogus  spectral sequence and Kato complexes yield a long exact sequence in cohomology (see   \S4 in \cite{igor_finiteness})
\begin{align*}
\cdots \rightarrow H^{i}_{\acute{e}t}(C, \mu_l^{\otimes i-1}) \rightarrow H^i(K,  \mu_l^{\otimes i-1}) \xrightarrow{r^i_l} \bigoplus_{v \in V} H^{i-1}(K^{(v)}, \mu_l^{\otimes i-2}) \rightarrow H^{i+1}_{\acute{e}t}(C, \mu_l^{\otimes i-1}) \rightarrow \cdots
\end{align*}
By Lemma \ref{lem:prime} and Corollary 3.2 in \cite{igor_finiteness},  $H^{i}_{\acute{e}t}(C, \mu_l^{\otimes j})$ is finite for all $i \geq 0$ and all $j$. This proves the claim.
\end{proof}

The Main Theorem now follows from Proposition \ref{prop:hypothesis} and Theorem \ref{thm:main}.
\\

The paper is organized as follows. In \S\ref{sec:bigprelim} we briefly recall the necessary results from the  literature that will lay foundation for the rest of the paper. In \S\ref{sec:unramified}, we define the notion of good reduction of skew-hermitian forms and relate it to good reduction of the universal covering of special unitary groups of skew-hermitian forms. In \S\ref{sec:morita} we use Morita theory to reduce skew-hermitian forms over quaternions to quadratic forms and give an explicit example of the correspondence which will be used later.   Next in \S\ref{sec:ext} we describe a method to extend valuations from a  discrete valued field to the function field of Severi-Brauer variety associated to a quaternion algebra over the field and discuss the properties of this extension.  Finally, in \S\ref{sec:proofmain}, we prove Theorem \ref{thm:main}.

\section{Preliminaries} \label{sec:bigprelim}
\subsection{Fields of type ($\text{F}_m'$)}\label{sec:fm}
The notion of fields of type ($\text{F}_m'$) is introduced in \cite{igor_finiteness}. Recall from \cite{igor_finiteness} that for  $m$ prime to $char ~k$, a field $k$ is said to be \emph{of type} ($\text{F}_m'$) if for every finite separable extension  $L/k$, the quotient $L^*/(L^*)^m$ is finite (here $L^*$ is the multiplicative group of units). This notion generalizes Serre's condition (F) (\cite{serre_galois_coho}) and is useful for many applications. It is shown that over fields of type ($\text{F}_m'$), certain Galois cohomology groups are finite (see Theorem 1.1 in  \cite{igor_finiteness}), which is useful for computations of  unramified cohomologies (Proposition 4.2 in \cite{igor_finiteness}).  Examples of such fields  include  finite fields, local fields and  higher dimensional local fields such as $\mathbb{Q}_p((t_1))\cdots ((t_n))$(note that the last two are not finitely generated). See also Example 2.9 in \cite{igor_finiteness}.  We now make the following observation (I thank P. Deligne to remark about this). 
\begin{lem} \label{lem:prime}
Let $k$ be a field and let $m$ be prime to its characteristic. Then $k$ is of type ($\text{F}_m'$) if and only if it is of type ($\text{F}_p'$) for every prime $p$ dividing $m$.
\end{lem}
\begin{proof}
By definition, it is clear that if $k$ is of type ($\text{F}_m'$), it is of type ($\text{F}_p'$) for every prime $p$ dividing $m$. We will prove the other direction by induction on the number of primes dividing $m$. Let $r$ be the  number of primes dividing $m$. Consider the case $r=1$. Assume that $k$ is of type ($\text{F}_p'$).  Let $L$ be a finite separable extension of $k$. For every $j \geq 1$ we have exact sequences of groups
\begin{align*}
0 \rightarrow \frac{(L^* )^{p^{j-1}}}{(L^* )^{p^j}}  \rightarrow \frac{(L^*)}{(L^*)^{p^j}} &\rightarrow \frac{(L^ *)}{(L^*)^{p^{j-1}}} \rightarrow 0\\
x &\mapsto x\\
\end{align*}
\begin{align*}
0 \rightarrow \frac{\mu_{p^{j-1}}(L^*)}{\mu_{p^{j-1}}(L^*) \cap (L^*)^p} \rightarrow \frac{(L^ *)}{(L^*)^p} &\rightarrow  \frac{(L^* )^{p^{j-1}}}{(L^* )^{p^j}} \rightarrow 0\\
x &\mapsto x^{p^{j-1}}
\end{align*}
where $\mu_n(L^*)$ denotes $n$-th roots of unity in $L^*$. By hypothesis, these sequences  imply that  $\frac{(L^* )}{(L^* )^{p^j}}$ is finite if $\frac{(L^* )}{(L^* )^{p^{j-1}}}$ is finite. Thus by induction on $j$ we conclude that  $k$ is of type ($\text{F}_{p^j}'$).  Therefore the the statement of the lemma is true for $r=1$. Assume that $m$ is arbitrary and $k$ is of type ($\text{F}_p'$) for every prime $p$ dividing $m$. Let $m= p^jn$  where $n$ is coprime to $p$ and  $j \geq 1$.   We have the  following exact sequences 
\begin{align*}
0 \rightarrow \frac{(L^*)^{p^j}}{(L^*)^m} \rightarrow \frac{(L^*)}{(L^*)^m} &\rightarrow \frac{(L^*)}{(L^*)^{p^j}} \rightarrow 0\\
x&\mapsto x \\\\
0 \rightarrow \frac{\mu_{p^j}(L^*)}{\mu_{p^j}(L^*) \cap (L^*)^n} \rightarrow \frac{(L^ *)}{(L^*)^n} &\rightarrow  \frac{(L^*)^{p^j}}{(L^* )^m} \rightarrow 0\\
x &\mapsto x^{p^j}
\end{align*}
By induction hypothesis on $r$, $\frac{(L^ *)}{(L^*)^n}$ is finite and by the case $r=1$, $\frac{(L^*)}{(L^*)^{p^j}}$ is finite. Therefore we conclude that $\frac{(L^*)}{(L^*)^m}$ is finite.  This proves that  $k$ is of type ($\text{F}_m'$). 
\end{proof}

 This settles the query raised in the statement below Conjecture 7.3 in \cite{rapinchuk_spinor}. 

\subsection{Residue maps and ramification}\label{sec:unramified}
Let $K$ be a field with discrete valuation $v$ and let $l$ be prime to $char ~ K^{(v)}$. Recall from Chapter II in \cite{garibaldi_coho_inv} that  for every integer $j$ we have residue maps in Galois cohomology
\begin{align}
r_{l,v}^{j}: H^i(K, \mu_l^{\otimes j}) \rightarrow H^{i-1}(K^{(v)}, \mu_l^{\otimes j-1}) , ~i \geq 1
\end{align}
where $\mu_l$ is the group of  $l$-th roots of unity in the separable closure of $K^{(v)}$ and  $\mu_l^{\otimes j}$ is the $j$-th Tate twist of $\mu_l$ as described in \cite{garibaldi_coho_inv} (Chapter II, \S7.8). An element of $H^i(K, \mu_l^{\otimes j})$ is said to be \emph{unramified at $v$} if  is in the kernel of the above residue map. Now assume that  $char ~K^{(v)} \neq 2$ and  $l=2$. Then we simply have 
\begin{align*}
r_{2,v}: H^i(K, \mu_2) \rightarrow H^{i-1}(K^{(v)}, \mu_2)
\end{align*}
Let $\pi$ be a uniformizer of $K_v$. Let  $W(K)$ denote  the Witt ring of $K$. Recall from \cite{lam_quadratic}, Chapter VI, \S1 that we have  residue homomorphisms of groups
\begin{align*}
\partial_{i,v} : W(K) \xrightarrow{Res_{K_v/K}} W(K_v) \rightarrow W(K^{(v)})
\end{align*}
The  residue homomorphisms can be  described as follows. Let  $q$ be a quadratic form over $K$  and let  [q] denote its class in $W(K)$. Suppose $Res_{K_v/K} ([q]) = <u_1, u_2 , \cdots u_m, \pi u_{m+1} \cdots \pi u_{n}>, u_i \in \mathcal{O}_{K_v}^*$. Then 
\begin{align*}
\partial_{1,v}([q]) &= <\overline{u}_1, \cdots , \overline{u}_m> \\
\partial_{2,v}([q]) &= <\overline{u}_{m+1}, \cdots, \overline{u}_n>
\end{align*}
Here $\overline{u_i}$ denotes the image of $u_i$ in  $K^{(v)}$. When $K= K_v$, let $W_0(K_v)$  denote the kernel of $\partial_{2,v}$.  It is the subring  of $W(K_v)$ generated by classes $<u>, u \in \mathcal{O}_{K_v}^*$.   Then we have a split exact sequence (see \S5 in \cite{milnor_quad})
\begin{align*}
0 \rightarrow W_0(K_v) \rightarrow W(K_v) \xrightarrow{\partial_{2,v}} W(K^{(v)})\rightarrow 0
\end{align*}

 Recall now that  due to Voevodsky's proof of  the Milnor conjecture (\cite{voe1}, \cite{voe2}), for any field $F$ with characteristic $\neq 2$ we have natural  isomorphisms 
\begin{align} \label{eqn:milnor}
e_ n: \mathcal{I}(F)^n/ \mathcal{I}(F)^{n+1} \rightarrow H^n(F, \mu_2)
\end{align}
where $\mathcal{I}(F)$ denotes the fundamental ideal in $W(F)$. Moreover, the isomorphisms $e_n$ commute with the respective residue homomorphisms (see Satz 4.11 in \cite{arason}), that is for $[q] \in  \mathcal{I}(K)^n$, we have  
\begin{align*}
e_{n-1}(\partial_{2,v}([q])) = r_{2,v}(e_n([q]))
\end{align*}
  \begin{defn}
  We say that $q$ is unramified at  $v$ if $[q] \in W_0(K_v)$ (There is  a slight abuse of notation here to make it look tidy, what we really mean is  $Res_{K_v/K}([q])  \in W_0(K_v)$). 
  \end{defn}
Let us denote the kernel of the map  (see \S10 in \cite{saltman_division} and \cite{serre_galois_coho} Chapter II, Appendix) 
\begin{align} \label{eqn:residue}
{}_{2}Br(K) \simeq H^2(K, \mu_2) \xrightarrow{r_v} H^1( K^{(v)} , \mu_2)
\end{align}
 by ${}_{2}Br(K)_{v}$.  For a quaternion  $Q$  over $K$, if $[Q]$ is in the kernel of the above map, we say that $Q$ is  \emph{unramified}  at $v$. Let $Q_v:= Q \otimes_K K_v$.  Then $Q_v$ is either split i.e, a matrix algebra  or is a quaternion division algebra over $K_v$.  Suppose $Q_v$ is not split.   Since $K_v$ is Henselian, one can extend the valuation $v$ on $K_v $ to a (necessarily unique) valuation on $Q_v$ (Corollary 2.2  in \cite{wadsworth_valuation_division}),  which by abuse of notation is  also denoted by $v$. The extended valuation on $Q_v$ is   given by (equation (2.7) in  \cite{wadsworth_valuation_division})
\begin{align} \label{eqn:extval}
v(a) = \frac{1}{2}v(Nrd(a)) ~ \forall a \in Q_v^*
\end{align}
where $Nrd$ denotes the reduced norm on $Q_v$.  Let 
\begin{align*}
A_v= \{ a \in Q_v^* ~| ~v(a) \geq 0 \} \cup 0
\end{align*}
be the valuation ring of $Q_v$.  Note that its group of units is given by 
\begin{align*} 
A_v^* =  \{ a \in Q_v^* ~| ~v(a) = 0 \}
\end{align*}
If $Q_v\simeq M_2(K_v)$ is split, we set $ A_v := M_2(\mathcal{O}_v)$.  By Theorem 10.3 in \cite{saltman_division} and Theorem 3.2 in \cite{wadsworth_valuation_division},  if  $Q_v$ is unramified at $v$, then  $A_v$ is an Azumaya algebra over $\mathcal{O}_{v}$ and $Q_v \simeq A_v \otimes_{\mathcal{O}_{K_v}}  K_v$. Since $[A_v: \mathcal{O}_{K_v}] = [Q_v: K_v] = 4$, $A_v$  is quaternionic and has representation given by $A_v = (d_v, t_v)$ for some $d_v, t_v \in \mathcal{O}_{K_v}^*$. (See  Theorem 3.2 in \cite{wadsworth_valuation_division} and Example 2.4 (ii) and Proposition  2.5 in \cite{jacob_wadsworth}). \\
\indent Recall now the following exact sequence (see Prop 7.7  in \cite{garibaldi_coho_inv}, \S3 in \cite{wadsworth_valuation_division} and Theorem 2, \S 3 Chapter XII in \cite{serre_local_fields})
\begin{align}\label{eqn:exact}
0 \rightarrow H^2(K^{(v)}, \mu_2) \xrightarrow{s} H^2(K_v, \mu_2) \xrightarrow{r} H^1(K^{(v)}, \mu_2) \rightarrow 0
\end{align}
 where $r$ is the residue map and $s$ is the canonical map resulting from $\mathcal{G}_{K_v} \rightarrow \mathcal{G}_{K^{(v)}}$(Here  $\mathcal{G}_F$ denotes the absolute Galois group of a field $F$). This yields an isomorphism (see equation (3.7) in \cite{wadsworth_valuation_division})
\begin{align*}
s^{-1}: ker(r) : = {}_{2}Br(K_v)_{v} &\xrightarrow{\simeq} {}_{2}Br(K^{(v)})\\
Q_v &\mapsto Q^{(v)}
\end{align*}
 where $Q^{(v)}$ is the \emph{residue quaternion algebra} given by $(\overline{d}_v, \overline{t}_v)$ (Here   $\overline{d}_v, \overline{t}_v \in K^{(v)}$ are the residues  obtained by taking  the quotients of  $d_v, t_v$ modulo the maximal ideal in $\mathcal{O}_{K_{v}}$).\\
  
\begin{defn}
Let $(h,Q)$ be a non-degenerate skew-hermitian form over $Q$ and let $(h_v,Q_v):=(h, Q)\otimes_K {K_v}$  be the form over $Q_v$ obtained via base change.  We say that $(h,Q)$ has  good reduction at $v$ if $h_v$ is obtained via base change from a  non-degenerate skew-hermitian form $\tilde{h}$ over the Azumaya algebra $A_v$   i.e, $(h_v, Q_v)\simeq (\tilde{h}, A_v) \otimes_{\mathcal{O}_{K_v}} K_v$. 
\end{defn}
\begin{rmk} \label{rmk:unr}
 By Theorem 10.3 in \cite{saltman_division}, if  $(h,Q)$ has good reduction at $v$ then $Q$ is unramified at $v$.
\end{rmk}

\begin{rmk} \label{rmk:equiv}
The universal covering of $SU(h,Q)$ has good reduction at $v$ if and only if the form $(\alpha_v h_v, Q_v)$  has good reduction   at $v$ for some $\alpha_v \in K_v^*$. (The "if" direction is clear. For the "only if" direction use the classification from \cite{srimathy_azumaya}  and the equivalence between Azumaya algebras with involutions and hermitian spaces from \S2.2 in \cite{sofie_thesis})
\end{rmk}

\section{Reduction to Quadratic Forms via Morita Equivalence} \label{sec:morita}
As before, let $Q$  denote a (not necessarily division) quaternion algebra with center $K$.
\subsection{General theory}\label{subsec:general}
The general theory of Morita equivalence for Hermitian modules can be found in Knus' book \cite{knus_book} (See Chapter 1, \S9). In particular, by Morita theory, a non-degenerate skew- hermitian form of rank $n$ over $Q$  gives rise to  a non-degenerate quadratic form of rank $2n$ over  $K$ whenever  $Q$ is split.  In this case, let us denote the quadratic form associated to the skew-hermitian from $h$ by $q_h$.  By the properties of Morita equivalence, $h$ is determined by $q_h$ and moreover, two  such skew-hermitian forms are isometric if and only if the associated quadratic forms are isometric. So whenever $Q$ is split the skew-hermitian forms over $Q$ can be completely studied by studying the associated quadratic forms over $K$. For an explicit description of Morita equivalence in this case see  \cite{scharlau_book}, p. 361-362. \\
\indent Let $h$ be  a skew-hermitian form over  a non-split $Q$.  A generic way to split  $Q$  is by extending the base field to the function field of the associated Severi-Brauer variety.  Let $K(Q)$ denote the function field of the Severi-Brauer variety associated to $Q$.  Now $Q_{K(Q)}$ is isomorphic to the matrix algebra $M_2(K(Q))$ with involution given by 
\begin{align*}
M \mapsto
 \begin{bmatrix}
0 &1\\
-1 & 0
\end{bmatrix}
M^t
\begin{bmatrix}
0 & 1\\
-1 & 0
\end{bmatrix} ^{-1}
\end{align*}
Since $Q_{K(Q)}$ is split, the skew-hermitian form $h_{K(Q)}:=h \otimes_K K(Q)$  can be reduced to a quadratic form $q_{h_{K(Q)}}$ via Morita equivalence. This reduction has nice properties due to the following  result from \cite{parimala_analogue}.
\begin{prop} (Proposition 3.3 in \cite{parimala_analogue}) 
Let $W^{-1}(Q)$ denote  the Witt group of skew-hermitian forms over $Q$. With the notations as above, the  canonical homomorphism 
\begin{align*}
W^{-1} (Q) \rightarrow W^{-1}(Q \otimes_K K(Q))
\end{align*}
is injective.
\end{prop}
 This means that $h$ is hyperbolic if and only if $h_{K(Q)}$ is hyperbolic if and only if, by Morita equivalence, $q_{h_{K(Q)}}$ is hyperbolic. This philosophy of  understanding  the skew-hermitan form $h$  over $Q$ by studying the quadratic form $q_{h_{K(Q)}}$ is cleverly employed in Berhuy's paper \cite{berhuy_skew}  to find cohomological invariants. We will be using this philosophy to study reduction properties of these forms.

\subsection{An explicit example}\label{sec:example}
\begin{example}\label{example} \normalfont As before let $Q =(d, t)$ be a quaternionic (not necessarily division)  algebra over $K$ with basis $<1, i, j, ij ~|~ i^2=d, j^2 =t, ij =-ji>$. Then the Severi-Brauer variety of $Q$ has function field  $K(Q)$ given by the fraction field of  $K[x,y]/(dx^2+ty^2-1)$.  An explicit splitting of $Q$ over $K(Q)$ is given by the following.
\begin{align*}
Q \otimes_K K(Q) &\xrightarrow{\simeq} M_2(K(Q)) \\
i & \mapsto \begin{bmatrix}
 dx& -y\\
-dty & -dx
\end{bmatrix}\\
j &\mapsto \begin{bmatrix}
 ty& x\\
dtx & -ty
\end{bmatrix}\\
ij&\mapsto \begin{bmatrix}
 0&  1\\
-dt & 0
\end{bmatrix}
\end{align*}

Let $h$ be a non-degenerate skew-hermitian form over $Q$ of rank $n$. Then it is well-known that $h$ has a diagonal matrix representation over $Q$ (see for example \S6 in \cite{lewis}).  By abuse of notation, let us denote the matrix also by $h$. Since $h$ is skew-hermitian, the diagonal entries are pure quaternions.  Let 
\begin{align*}
h \simeq \bigoplus_{l=1}^n  a_l i + b_l j + c_l ij,~~~~~~ a_l, b_l, c_l \in K
\end{align*} 
Then,
\begin{align*}
h_{K(Q)} \simeq  \bigoplus_{l=1}^n a_l \begin{bmatrix}
 dx& -y\\
-dty & -dx
\end{bmatrix} 
+
b_l \begin{bmatrix}
 ty& x\\
dtx & -ty
\end{bmatrix}
+
c_l \begin{bmatrix}
 0&  1\\
-dt & 0
\end{bmatrix}
\end{align*}
Now we use the explicit description of Morita equivalence from  \cite{scharlau_book}, p. 361-362 to conclude that the quadratic form associated to $h_{K(Q)}$ has matrix given by (again by abuse of notation)
\begin{align*}
q_{h_{K(Q)}} \simeq  \bigoplus_{l=1}^n a_l \begin{bmatrix}
 -dty& -dx\\
-dx & y
\end{bmatrix} 
+
b_l \begin{bmatrix}
 dtx& -ty\\
-ty& -x
\end{bmatrix}
+
c_l \begin{bmatrix}
 -dt&  0\\
0 & -1
\end{bmatrix}
\end{align*}  

Let $N_l = Nrd_Q(a_l i+ b_l j  + c_l ij) \in K$ denote the reduced norm  of the quaternion $(a_l i+ b_l j  + c_l ij)$ in $Q$.  Then diagonalizing the above matrix yields
\begin{align*}
q_{h_{K(Q)}} \simeq \bigoplus_{l=1}^n \begin{bmatrix}
(a_l y  - b_lx - c_l) & 0\\
0 &  -(a_l y - b_lx - c_l )N_l
\end{bmatrix}
\end{align*}
We will be using this matrix representation of $q_{h_{K(Q)}}$ later.
\end{example}

\begin{rmk}\label{rmk:similar}
From the above description of Morita equivalence, it is clear that for $\lambda \in K^*$, 
\begin{align*}
q_{(\lambda  h)_{K(Q)}} = \lambda q_{h_{K(Q)}}
\end{align*}
\end{rmk}

\begin{prop}\label{prop:goodred}
Let $h$ be a non-degenerate skew-hermitian form over $Q$. If $h$ has good reduction at $v$ then  $Q$ is unramified at $v$ and 
\begin{enumerate}[(i)]
\item $q_{h_v}$ is unramified at $v$ if $Q_v$ is split or
\item $h_v$ has a diagonal matrix representation  with diagonal entries taking values in $A_v^*$ if $Q_v$ is not split. 
\end{enumerate}
\end{prop}
\begin{proof}
The case when $Q_v$ is split is clear by Morita theory. When $Q_v$ is not split, the claim follows by observing that  $A_v$ has no zero divisors and hence any non-degenerate skew-hermitian form over $A_v$ has a diagonal representation with units along the diagonal (see Proposition 6.8 and Proposition 3.2 in \cite{sofie}).
\end{proof}

\section{Extension of valuation from $K$ to $K(Q)$} \label{sec:ext}
In this section assume that $Q$ is a quaternion \emph{division} algebra over $K$ unramified at  $v$. We first extend the valuation $v$ from $K_v$ to $K_v(Q_v)$. There are two cases: 
\begin{enumerate}[(i)]
\item  $Q_v$ is not split.  Then as already discussed in \S \ref{sec:unramified},  $Q_v \simeq A_v \otimes_{\mathcal{O}_{K_v}}  K_v$ where $A_v = (d_v,t_v)$ is a quaternionic Azumaya algebra over  $\mathcal{O}_{K_v}$ with $d_v, t_v \in  \mathcal{O}_{K_v}^*$ (hence $v(d_v) = v(t_v) =0$).  We extend the valuation $v$  on $K_v$ to $K_v(Q_v)$ as follows. First we extend the valuation $v$ from $K_v$ to the valuation $v'$ on $K_v[x]$ by
\begin{align} \label{eqn:ext1}
v'(\sum_{l=0}^{m} a_l x^l) = min\{ v(a_0), \cdots, v(a_m)\} 
\end{align}
and then extend to $K_v(x)$ by
\begin{align} \label{eqn:ext2}
v'(\frac{f}{g}) = v'(f) -v'(g)
\end{align}
This is indeed a valuation on $K_v(x)$ with residue field $K^{(v)}(\overline{x})$, where $\overline{x}$ is the residue of $x$ at $v'$ and ramification index $e_{v'/v} = 1$ (see Example 2.3.3 in \cite{fried_field_arithmetic}). \\
 \indent Now  $K_v(Q_v) = K_v(x)[y]/(d_vx^2+t_vy^2-1)$  is a quadratic extension of $K_v(x)$ . Let $\tilde{v}$  denote any valuation on $K_v(Q_v)$ extending the one on $K(x)$ given above (One can always extend valuations to a larger field by Theorem 4.1 in \cite{lang_algebra}). 
\item  $Q_v$ is split. In this case $K_v(Q_v) \simeq  K_v(x)$ where $x$ is transcendental. In this case we extend the valuation on $K_v(Q_v)$ using (\ref{eqn:ext1}) and (\ref{eqn:ext2}).
\end{enumerate}

\begin{prop}\label{prop:properties}
The valuation in $\tilde{v}$ on $K_v(Q_v)$ defined above has the following properties.
\begin{enumerate}[(i)]
\item The residue field  of the valuation $\tilde{v}$, denoted by  $K_v(Q_v)^{(\tilde{v})}$ is  isomorphic to $K^{(v)}(Q^{(v)})$,  the function field of the Severi-Brauer variety associated to the residue division algebra $Q^{(v)}$ over $K^{(v)}$.
\item The ramification index $e_{\tilde{v}/v} =1$.
\item The valuation $\tilde{v}$  is the unique one extending $v'$. In particular, for $\alpha \in  K_v(Q_v)$
\begin{align*}
\tilde{v} (\alpha) = \frac{1}{2} v'(Norm_{K_v(Q_v)/K_v(x)}(\alpha)) 
\end{align*}
\item For $a_i \in \mathcal{O}_{K_v}$, $\tilde{v}(a_1y + a_2x +a_3) = 0$ if and only if $min\{v(a_i)\} = 0$.
\end{enumerate}
\end{prop}
\begin{proof}
All of the above claims are clear when $Q_v$ splits. So assume that $Q_v$ is not split. 
To prove (i), note that $\tilde{v}(y^2) = v'(\frac{1}{t_v} - \frac{d_v}{t_v}x^2) =0$ since $d_v, t_v$ are units in $\mathcal{O}_{K_v}$. Hence $\tilde{v}(y) =0$ and $y \in \mathcal{O}_{\tilde{v}}^*$.  Let $\overline{y}, \overline{x}, \overline{d_v}, \overline{t_v}$ denote the corresponding residues at $\tilde{v}$. Then we have an embedding,
\begin{align*}
F:=K^{(v)}(\overline{x})[\overline{y}]/(\overline{d}_v\overline{x}^2+\overline{t}_v\overline{y}^2-1) \hookrightarrow K_v(Q_v)^{(\tilde{v})}
\end{align*}

Now $(\overline{d}_v\overline{x}^2+\overline{t}_v\overline{y}^2-1)$ is the conic corresponding to the residue quaternion algebra $Q^{(v)}$. So $F\simeq K^{(v)}(Q^{(v)})$. Moreover $Q^{(v)}$ is not split because $Q$ is unramified at $v$ and due to injectivity of $s$ in the exact sequence of  (\ref{eqn:exact}) in \S \ref{sec:unramified}. Therefore the conic $(\overline{d}_v\overline{x}^2+\overline{t}_v\overline{y}^2-1)$ is not hyperbolic over $K^{(v)}(\overline{x}$) and  hence $ [K^{(v)}(Q^{(v)}) : K^{(v)}(\overline{x})] = 2$. But since $[K_v(Q_v)^{(\tilde{v})}:K^{(v)}(\overline{x})] \leq [K_v(Q_v): K_v(x)] =2$,  we conclude that  $K_v(Q_v)^{(\tilde{v})}$ is  isomorphic to $K^{(v)}(Q^{(v)})$. \\
\indent We now prove (ii) and (iii). Let $g$ be the number of distinct valuations on $K_v(Q_v)$ extending $v'$.  From the above argument we see that  $f_{\tilde{v}/v'} =[K_v(Q_v)^{(\tilde{v})}: K^{(v)}(\overline{x})] =2$. This implies $e_{\tilde{v}/v'} =1$ and  $g=1$ from the equality
\begin{align*}
[K_v(Q_v): K_v(x)] = e_{\tilde{v}/v'} f_{\tilde{v}/v'} g
\end{align*}
Also as mentioned before $e_{v'/v} =1$ (from Example 2.3.3 in \cite{fried_field_arithmetic}). This proves that $e_{\tilde{v}/v} =1$. From the fact the Galois group $G(K_v(Q_v)/K_v(x))$ acts transitively on the extensions on the valuation $v'$ on  $K_v(Q_v)$ (Exercise 8, Chapter 2 in \cite{fried_field_arithmetic}) and   $g=1$, we get (iii).  \\
\indent Note that by (iii) we have 
\begin{align*}
\tilde{v}(a_1y + a_2x + a_3) &= \frac{1}{2} v'(Norm _{K_v(Q_v)/K_v(x)}(a_1y + a_2x + a_3))\\
&= \frac{1}{2}v'(-a_1^2(\frac{1}{t_v} - \frac{d_v}{t_v}x^2) + a_2^2x^2 + a_3^2 + 2a_2a_3x)\\
&=\frac{1}{2}v'((a_2^2 +\frac{d_v}{t_v}a_1^2)x^2 + 2a_2a_3 x + (a_3^2 - \frac{1}{t_v}a_1^2))
\end{align*}
From this (iv) easily follows.
\end{proof}

By abuse of notation, let $\tilde{v}$ also denote the valuation on $K(Q)$ obtained by restriction via the  embedding $K(Q) \hookrightarrow K(Q) \otimes_K  K_v \simeq K_v(Q_v)$.\\
\indent We now prove Theorem \ref{thm:cool} that relates good reduction of a   skew-hermitian form $h$ over $Q$  with center $K$  to the ramification of the associated  quadratic form $q_{h_{K(Q)}}$.
\\
\\
\emph{Proof of Theorem \ref{thm:cool}:}
Given a discrete valuation $v$ on $K$, let $\tilde{v}$ be the valuation on $K(Q)$ described as above. By hypothesis, $h$ has good reduction at $v$. So  by Remark \ref{rmk:unr},  $Q$ is necessarily unramified at $v$.  We need to show that  $q_{h_{K(Q)}}$ is unramified  $\tilde{v}$. There are two cases:
\begin{enumerate}[(i)]
\item $Q_v$ is split. Since $K(Q)_{\tilde{v}}$ contains $K_v(Q_v)$ as a subfield, it suffices to show that $q_{h_{K_v(Q_v)}}$ is unramified at $\tilde{v}$. But  by functoriality of Morita equivalence, we have  $q_{h_{K_v(Q_v)}} = q_{h_v} \otimes_{K_v} K_v(Q_v)$. Together with  Proposition \ref{prop:goodred}(i),  we get that $q_{h_{K(Q)}}$ is unramified  $\tilde{v}$.
\item $Q_v$ is not split. With notation as  in \S\ref{sec:unramified}, $Q_v = A_v \otimes_{\mathcal{O}_{K_v}} K_v$, where $A_v$, the valuation ring of $Q_v$,  is a quaternionic Azumaya algebra over $\mathcal{O}_{K_v}$ given by   $A_v = <1, i, j , ij ~|~ i^2 =d_v, j^2 = t_v, ij = -ji ,d_v, t_v \in A_v^*>$. By Proposition \ref{prop:goodred}(ii), $h_v$ has a matrix representation that is diagonal with diagonal entries taking values in $A_v^*$. Consider one such representation
\begin{align*}
h_v \simeq \bigoplus_{l=1}^n  a_l i + b_l j + c_l ij
\end{align*}
where $ a_l i + b_l j + c_l ij  \in A_v^*$. Let $N_l = Nrd_Q(a_l i+ b_l j  + c_l ij)$ be the reduced norm. Then by (\ref{eqn:extval}),
\begin{align*}
0&=v(a_l i + b_l j + c_l ij) \\
&= \frac{1}{2}v(N_l )\\
&= \frac{1}{2}v(-a_l^2d_v -b_l^2t_v +c_l^2d_vt_v)
\end{align*}
This implies that  for each $l$ that $min\{v(a_l), v(b_l), v(c_l)\}  = 0$. Now  recall that by Example \ref{example} in \S\ref{sec:example}, 
\begin{align*}
q_{h_{K_v(Q_v)}} \simeq \bigoplus_{l=1}^n \begin{bmatrix}
(a_l y  - b_lx - c_l) & 0\\
0 &  -(a_l y - b_lx - c_l )N_l
\end{bmatrix}
\end{align*}
We are now done by Proposition \ref{prop:properties}(iv).
\end{enumerate}

\section{Proof of Theorem \ref{thm:main}}\label{sec:proofmain}
We begin with an easy lemma. We stick to the notations as before.
\begin{lem}\label{lem:residue}
Let $Q$ be  unramified at  $v$. Then we have 
\begin{align*}
r_2^i([Q] \cup H^{i-2}(K, \mu_2)) \subseteq [Q^{(v)}] \cup H^{i-3}(K^{(v)}, \mu_2)
\end{align*}
where $r_2^i$ is the residue map given by  (\ref{eqn:res2}) in \S\ref{sec:main}.
\end{lem}
\begin{proof}
 Since $Q$ is unramified at $v$, $r_2^i([Q]) =0$. Therefore by the exact sequence (see Proposition 7.7 in \cite{garibaldi_coho_inv})
\begin{align*}
0 \rightarrow H^i(K^{(v)}, \mu_2) \xrightarrow{s_2^i} H^i(K_v, \mu_2) \xrightarrow{r_2^i}H^{i-1}(K^{(v)}, \mu_2) \rightarrow 0
\end{align*}
we have that $[Q_v]$ is uniquely identified as the image of the residue quaternion algebra $[Q^{(v)}]$ over $K^{(v)}$ under $s_2^i$. The result now follows from Chapter II \S7, Exercise 7.12 in \cite{garibaldi_coho_inv}.
\end{proof}
Now let us  recall some of the facts discussed in \S 1.2.2 of Berhuy's paper \cite{berhuy_skew}. To simplify notations, let $F(Q)$ be the function field of the Severi-Brauer variety associated to a quaternion algebra $Q$ over an arbitrary field $F$. Now consider the valuations $W$ on $F(Q)$ arising from closed points on the conic defined by $Q$. Then for every $l \geq 1$, the kernel of the corresponding residue map 
\begin{align*}
r_l^i:  H^i(F(Q), \mu_l^{\otimes i-1}) &\rightarrow \bigoplus_{w \in W} H^{i-1}(F(Q)^{(w)}, \mu_l ^{\otimes i-2})
\end{align*} 
is the unramified cohomology group with respect to the valuations in $W$ denoted by $H^i_{nr}(F(Q), \mu_l^{\otimes i-1})$.  We now let 
\begin{align*}
\mathbb{Q}/\mathbb{Z} (i-1) = \lim_{\rightarrow} \mu_l^{\otimes{i-1}}
\end{align*}
where the limit is taken over all the integers prime to the characteristic of $F$.  Then $H^{i}(F(Q), \mathbb{Q}/\mathbb{Z} (i-1))$ is the  direct limit of the groups  $H^i(F(Q), \mu_l^{\otimes i-1}) $ with respect to the maps 
\begin{align*}
H^i(F(Q), \mu_m^{\otimes i-1}) \rightarrow H^i(F(Q), \mu_n^{\otimes i-1}) , m|n
\end{align*}
The corresponding residue maps are compatible to each other yielding the unramified cohomology $H^i_{nr}(F(Q), \mathbb{Q}/\mathbb{Z}(i-1))$.   Moreover for any field $E$, $H^{i}(E, \mu_{2^m} ^{\otimes i-1})$ is identified with the $2^m$ torsion subgroup of $H^i(E, \mathbb{Q}/\mathbb{Z}(i-1))$ and hence the canonical map of change of coefficients
\begin{align} \label{eqn:injective}
H^{i}(E, \mu_{2^m} ^{\otimes i-1}) \hookrightarrow H^{i}(E, \mu_{2^n} ^{\otimes i-1}) ~~~m|n
\end{align}
is injective. We now recall the following results from  \cite{berhuy_skew}.
\begin{prop} \label{prop:berhuy} (Proposition 7 and Proposition 9 in \cite{berhuy_skew})
Let $h$ be a skew-hermitian form over a quaternionic $F$-algebra $Q$. Then 
\begin{enumerate}[(i)]
\item $e_i([q_{h_{F(Q)}}]) \in H^i_{nr}(F(Q), \mu_2)$
\item For $i \geq 1$, the restriction map yields an isomorphism  
\begin{align*}
Res_{F(Q)/F} : H^i(F, \mathbb{Q}/\mathbb{Z} (i-1) /([Q] \cup H^{i-2}(F, \mu_2)) \simeq H_{nr}^i(F(Q), \mathbb{Q}/\mathbb{Z}(i-1)) 
\end{align*}
where $[Q] \cup H^{i-2}(F, \mu_2)$ is viewed as a subgroup of $H^i(F, \mathbb{Q}/\mathbb{Z} (i-1))$ (If $i=1$, $[Q] \cup H^{i-2}(F, \mu_2) = 0$ ). In particular,  the inverse image  denoted by $S$, of   $2$-torsion subgroup of the unramified cohomology group  $H_{nr}^i(F(Q), \mu_2)$ under $Res_{F(Q)/F}$ is  a subgroup of $H^i(F, \mu_4^{\otimes{i-1}}) /([Q] \cup H^{i-2}(F, \mu_2))$. Thus we have  an isomorphism $\jmath$ obtained by restricting  $Res_{F(Q)/F}$ to $S$,
\begin{align*}
\jmath: S ~\xrightarrow{\simeq}  H_{nr}^i(F(Q), \mu_2)
\end{align*}

\end{enumerate}
\end{prop}

\noindent \textbf{Notation:} Let $\tilde{V}$ denote the collection of valuations $\tilde{v}$ on $K(Q)$ obtained by extending the valuations $v \in V$ on $K$ as described in \S\ref{sec:ext}.    To simplify notations, from now on  let  $L = K(Q)$.\\

We  now  prove that the hypothesis of Theorem \ref{thm:main} implies the following finiteness theorem.
\begin{thm} \label{thm:finite}
Let $(K,V)$ satisfy the hypothesis of Theorem \ref{thm:main}.  Then for each $i \geq 1$, the kernel of the residue map
\begin{align*}
R^i_2 : H^i_{nr}(L, \mu_2) \rightarrow \bigoplus_{\tilde{v} \in \tilde{V}} H^{i-1}(L^{(\tilde{v})}, \mu_2)
\end{align*}
denoted by $H^i_{nr} (L, \mu_2)_{\tilde{V}}$, is finite and is bounded by
\begin{align*}
|H^i_{nr} (L, \mu_2)_{\tilde{V}}| \leq |H^i (K , \mu_2)^{V}|\cdot |H^i(K, \mu_4^{\otimes i-1})_{V}|
\end{align*}
\end{thm}
\begin{proof}
For each $v \in V$, the following diagram  with arrows representing the natural restriction maps commutes by the functoriality.

\begin{equation*}
\begin{tikzpicture}[thick]
  \node (A) {$H^i(K, \mu_4^{\otimes{i-1}})$};
  \node (B) [right of=A] {$H^i(K_v, \mu_4^{\otimes{i-1}})$};
  \node (C) [below of=A] {$H^i(L, \mu_4^{\otimes{i-1}})$};
  \node (D) [below of=B] {$H^i(L_{\tilde{v}}, \mu_4^{\otimes{i-1}})$};

  \draw[->] (A) to node {} (B);
  \draw[->] (A) to node [swap]  {}(C);
  \draw[->] (C) to node [swap]  {}(D);
  \draw[->] (B) to node {} (D);
 
\end{tikzpicture}
\end{equation*}

Moreover by  Chapter II, \S8,  Proposition 8.2 in \cite{garibaldi_coho_inv},  the  functoriality of restriction yields
\begin{equation*}
\begin{tikzpicture}[thick]
  \node (A) {$H^i(K_v, \mu_4^{\otimes{i-1}})$};
  \node (B) [right of=A] {$H^{i-1}(K^{(v)}, \mu_4^{\otimes{i-2}})$};
  \node (C) [below of=A] {$H^i(L_{\tilde{v}}, \mu_4^{\otimes{i-1}})$};
  \node (D) [below of=B] {$H^{i-1}(L^{(\tilde{v})}, \mu_4^{\otimes{i-2}})$};

  \draw[->] (A) to node {} (B);
  \draw[->] (A) to node [swap]  {}(C);
  \draw[->] (C) to node [swap]  {}(D);
  \draw[->] (B) to node {$e_{\tilde{v}/v} =1$} (D);
 
\end{tikzpicture}
\end{equation*}
where the horizontal arrows represent the residue maps. Combining the above commutative diagrams for each $v \in V$, together with Proposition \ref{prop:berhuy}, Lemma \ref{lem:residue} and  (\ref{eqn:injective}), we get that the following diagram commutes where $\imath$ is injective and $\jmath$ is an isomorphism.
\begin{equation*}
\begin{tikzpicture}[thick] 
  \node (A) {\large $\frac{H^i(K, \mu_4^{\otimes{i-1}})}{([Q] \cup H^{i-2}(K, \mu_2))}$};
  \node (B) [right of=A] {\large $\bigoplus_{v\in V} \frac{H^{i-1}(K^{(v)}, \mu_4^{\otimes{i-2}})}{([Q^{(v)}] \cup H^{i-3}(K^{(v)}, \mu_2))} $};
  \node (C) [below of=A] {$H^i_{nr}(L, \mu_2)$};
  \node (D) [below of=B] {$\bigoplus_{\tilde{v} \in V} H^{i-1}(L^{(\tilde{v})}, \mu_2)$};
\node (E) [left of = A] {$S$};
\draw[->] (E) to node {$\jmath$} (C);
  \draw[->] (A) to node {$r^i_4$} (B);
\draw[->] (E) to node {$\imath$} (A);
  \draw[->] (C) to node [swap]  {$R^i_2$}(D);
  \draw[->] (B) to node {$e_{\tilde{v}/v} =1$} (D);
 
\end{tikzpicture}
\end{equation*}
Now  $L_{\tilde{v}}$ is also the completion of $K_v(Q_v)$ at $\tilde{v}$.  So by Proposition \ref{prop:properties}(ii), for the extension $L_{\tilde{v}}/K_v$, we have ramification index $e_{\tilde{v}/v} =1$ and  the residue field $L^{(\tilde{v})}  \simeq K^{(v)}(Q^{(v)})$. Hence the right vertical map is  injective by Proposition \ref{prop:berhuy}.   \\
\indent By the hypothesis on $(K, V)$, it is easy to see that  $Ker ~r_4^i$  is finite and 
\begin{align*}
|H^i_{nr}(L, \mu_2)_{\tilde{V}}|  = |Ker ~R_2^i| \leq  |Ker ~r_4^i| = |H^i (K , \mu_2)^{V}|\cdot |H^i(K, \mu_4^{\otimes i-1})_{V}|
\end{align*}

\end{proof}

We are now ready to prove Theorem \ref{thm:main}.\\

\noindent \emph{Proof of Theorem \ref{thm:main}} : \\
Since the unramified $2$-torsion Brauer group with respect to $V$,  ${}_{2}Br(K)_V$ is isomorphic to $H^2(K, \mu_2)_V$   which is finite by hypothesis, there are only finitely many quaternion algebras unramified at all $v \in V$. So it suffices to show that for a fixed unramified $Q$ over $K$, the number of $K$-isomorphism classes of the universal covering of the special unitary groups $SU_n(h, Q)$ of  $n$-dimensional skew-hermitian forms $h$ over $Q$  that  have good reduction at all $v \in V$is  upper bounded by 
\begin{align*}
 |Pic(V)/2 Pic(V)| \cdot  \prod_{i=1}^{l} |H^i (K , \mu_2)^{V}|\cdot |H^i(K, \mu_4^{\otimes i-1})_{V}|
\end{align*}
  We have two cases.
\begin{enumerate} [(i)]
\item $Q$ is a  split quaternion.  In this case $SU(h,Q) \simeq SO_{2n}(q_h, K)$. Then by  the proof of Theorem 2.1 in \cite{rapinchuk_spinor},  we conclude that the number of $K$-isomorphism classes of $SO_{2n}(q_h, K)$ that have good reduction at all $v \in V$ is finite and bounded above by 
\begin{align*}
|Pic(V)/2Pic (V)| \cdot \prod_{i=1}^l|H^i(K, \mu_2)_V| \leq  |Pic(V)/2 Pic(V)| \cdot  \prod_{i=1}^{l} |H^i (K , \mu_2)^{V}|\cdot |H^i(K, \mu_4^{\otimes i-1})_{V}|
\end{align*}
The above inequality is due to (\ref{eqn:injective}).
\item $Q$  is  a quaternionic division algebra unramified at all $v \in V$.  The idea of the proof in this case is  to go  back and forth between $h$ and $q_{h_{K(Q)}}$ and using arguments similar to the one in \cite{rapinchuk_spinor}. \\

 \noindent\textbf{Notation:} In order to avoid notational complexity, we will be simplifying some notations as follows. 
\begin{itemize}
\item For a skew-hermitian form $h$ over a  quaternion algebra $Q$ with center a  field $F$, we will make a slight abuse of notation and write $q_h$ instead of $q_{h_{F(Q)}}$ for the quadratic form corresponding to $h_{F(Q)}$ obtained via Morita theory. 
\item For a field $F$ and a  class $[q] \in \mathcal{I}(F)^m $, we denote by $<q> \in H^m(F, \mu_2)$, its image under the natural  map
\begin{align*}
\mathcal{I}(F)^m \rightarrow \mathcal{I}(F)^m/\mathcal{I}(F)^{m+1} \xrightarrow{\simeq} H^m(F, \mu_2)
\end{align*} 
where the first map is the natural projection and the second one is the Milnor isomorphism  as mentioned in  (\ref{eqn:milnor}).
\end{itemize} 
As before let $L:=K(Q)$. \indent Let $\{h_i\}_{i \in I}$ denote a family of $n$-dimensional non-degenerate skew-hermitian forms over $Q$ such that
\begin{itemize}
\item for each $i \in I$, the universal covering of $G_i = SU_n(h_i, Q)$ has good reduction at all $v \in V$ and 
\item for $i,j \in I, i \neq j$, the forms $h_i$ and $h_j$ are not similar i.e., $h_i \ncong\lambda h_j, \lambda \in K^*$.
\end{itemize}
It suffices to show that
\begin{align*}
|I|  \leq \prod_{i=0}^l d_i
\end{align*}
where  $d_0 = Pic(V)/2Pic(V)$ and for $1 \leq i \leq l = [log_2~2n] +1$,
\begin{align*}
 d_i &=  |H^i (K , \mu_2)^{V}|\cdot |H^i(K, \mu_4^{\otimes i-1})_{V}|
\end{align*}

Note that by Remark \ref{rmk:equiv}, the above conditions imply that for each $i \in I$ and any $v \in V$, there exists $\lambda_v^{(i)} \in K_v^*$ such that the form $\lambda_v^{(i)} h_i$ over $Q_v$ is has good reduction at $v$.  Also  because of condition  (A) on $K$, we can assume that  $ \lambda^{(i)}_v =1$ for almost all $v \in V$.  Recall by Lemma 2.2 in \cite{rapinchuk_spinor} that there is a natural isomorphism 
\begin{align}\label{eqn:ideles}
Pic(V)/2Pic(V) \simeq \mathbb{I}(K,V)/\mathbb{I}(K,V)^2\mathbb{I}_0(K,V)K^* 
\end{align}
where 
\begin{align*}
\mathbb{I}(K,V) =\big \{  (x_v) \in \prod_{v \in V}K_v^* ~|~ x_v \in \mathcal{O}_v^* \text{~for almost all~} v \in V     \big\}
\end{align*}
 is the \emph{group of id\`eles}  and 
\begin{align*}
\mathbb{I}_0(K,V) = \prod_{v \in V} \mathcal{O}_v^*
\end{align*}
is the \emph{subgroup of integral id\`eles}.
\end{enumerate}
So $\lambda^{(i)} : = (\lambda_v^{(i)})_{v \in V} \in \mathbb{I}(K,V)$. Since $d_0$ is finite by hypothesis,  using (\ref{eqn:ideles}), we conclude that  there exists a subset $J_0 \subseteq I$ of size  $\geq I/d_0$ (if $I$ is infinite so is $J_0$) such that all $\lambda^{(i)}, i \in J_0$ have the same image in $\mathbb{I}(K,V)/\mathbb{I}(K,V)^2\mathbb{I}_0(K,V)K^* $. Fix $j_0 \in J_0$.  For any $j \in J_0$, we can write 
\begin{align*}
\lambda^{(j)} = \lambda^{(j_0)} ( \alpha^{(j)})^2 \beta^{(j)} \delta^{(j)}
\end{align*}
with  $\alpha^{(j)} \in \mathbb{I}(K,V)$, $\beta^{(j)} \in \mathbb{I}_0(K,V)$ and $\delta^{(j)} \in K^*$. Then set 
\begin{align*}
H_j &= \delta^{(j)} h_j \\
\Lambda_j &= (\delta^{(j)})^{-1} \lambda^{(j)} = \lambda^{(j_0)} (\alpha^{(j)})^2 \beta^{(j)}
\end{align*}
It is easy to see that for $j\neq j'$, $H_j \ncong H_{j'}$ and hence $q_{H_j} \ncong q_{H_{j'}}$ as  quadratic forms over $L$ (see \S \ref{subsec:general}). Moreover, $\Lambda_{v}^{(j)} H_j = \lambda_v^{(j)}h_j$ and hence $\Lambda_{v}^{(j)} H_j $ has good reduction at $v$. Therefore by Theorem \ref{thm:cool} 
\begin{align*}
q_{\Lambda_v^{(j)}H_j} \in W_0(L_{\tilde{v}}) , ~j \in J_0
\end{align*}
(See \S\ref{sec:unramified} for the definition of $W_0(L_{\tilde{v}})$).
Also  note that 
\begin{align*}
q(j , \tilde{v}) &:= \Lambda_v^{(j_0)} ( q_{H_j} \perp  -q_{H_{j_0}} )\\
 &=\Lambda_v^{(j_0)} \cdot (\Lambda_v^{(j)})^{-1} \cdot \Lambda_v^{(j)} q_{H_j}  \perp  -\Lambda_v^{(j_0)}q_{H_{j_0}} \\
&=  (\alpha^{(j)})^{-2} (\beta^{(j)})^{-1} q_{\Lambda_v^{(j)}H_j}  \perp -q_{\Lambda_v^{(j_0)}H_{j_0}} ~~~~~~~~~~~~~~~~\text{(by Remark \ref{rmk:similar})}\\
&\cong(\beta^{(j)})^{-1} q_{\Lambda_v^{(j)}H_j}  \perp -q_{\Lambda_v^{(j_0)}H_{j_0}} 
\end{align*}
As $(\beta^{(j)})^{-1} \in \mathcal{O}_v^*$, we see that for every $v \in V$ 
\begin{align*}
[q(j , \tilde{v})] &= \Lambda_v^{(j_0)} ( [q_{H_j}] \perp  -[q_{H_{j_0}}] ) \in W_0(L_{\tilde{v}}) \cap \mathcal{I}(L_v)
\end{align*}
Now by Lemma 3.3 in \cite{rapinchuk_spinor} and Proposition \ref{prop:berhuy}, we get 
\begin{align*}
<q_{H_j}> - <q_{H_{j_0}}>  \in H^1_{nr}(L, \mu_2)_{\tilde{V}} \leq d_1
\end{align*}
The last inequality follows from Theorem \ref{thm:finite}. Therefore we can find a subset $J_1 \subseteq J_0$ of size $\geq |J_0|/d_1 \geq |I|/d_0d_1$ such that for $j  \in J_1$, the classes $[q_{H_j}] - [q_{H_{j_0}}] \in \mathcal{I}(L)$ all have the same image in $H^1(L, \mu_2)$. Fix $j_1 \in J_1$.  Then for any $j \in J_1$, we have 
\begin{align*}
<q_{H_j}> - <q_{H_{j_1}}> = (<q_{H_j}> - <q_{H_{j_0}}>) - (<q_{H_{j_1}}> - <q_{H_{j_0}}>) = 0 \in H^1(L, \mu_2)
\end{align*}
Hence  $[q_{H_j}] - [q_{H_{j_1}}]  \in \mathcal{I}(L)^2$. Moreover
\begin{align*}
\Lambda^{(j)} = \Lambda^{(j_1)}(\overline{\alpha}^{(j)})^2 (\overline{\beta}^{(j)}) ,  ~~~~~\overline{\alpha}^{(j)} \in \mathbb{I}(K,V), \overline{\beta}^{(j)} \in \mathbb{I}_0(K,V) 
\end{align*}
Then as before we  conclude that  for every $\tilde{v} \in\tilde{ V}$
\begin{align*}
\Lambda_v^{(j_1)} ( [q_{H_j}] \perp  -[q_{H_{j_1}}] ) \in W_0(L_{\tilde{v}}) \cap \mathcal{I}(L_{\tilde{v}})^2
\end{align*}
Again by using Lemma 3.3 in \cite{rapinchuk_spinor}, Proposition \ref{prop:berhuy} and Theorem \ref{thm:finite} as before,  we get
\begin{align*}
<q_{H_j}> - <q_{H_{j_1}}>  \in H^2_{nr}(L, \mu_2)_{\tilde{V}} \leq d_2
\end{align*}
So there exists a subset $J_2 \subseteq J_1$ of size  $\geq |I|/d_1d_1d_2$ such that that for each $j \in J_2$, the classes $[q_{H_j}] - [q_{H_{j_1}}]$ have the same image in $H^2(L, \mu_2)$. Fixing $j_2 \in J_2$, we have 
\begin{align*}
[q_{H_j}] - [q_{H_{j_2}}]  \in \mathcal{I}(L)^3 ~~~~j \in J_2
\end{align*}

Proceeding inductively  we get a nested chain  of subsets 
\begin{align*}
I \supseteq J_0 \supseteq J_1 \supseteq J_2 \supseteq \cdots \supseteq J_l
\end{align*}
such that for any $m =1, 2, \cdots l$, 
\begin{itemize}
\item $|J_m| \geq |I|/d_0d_1\cdots d_m$ and 
\item for $j \in J_m$, we have 
\begin{align*}
[q_{H_j}] - [q_{H_{j_m}}]  \in \mathcal{I}(L)^{m+1} ~~~~j \in J_m
\end{align*}
\end{itemize}
But by a theorem of Arason and Pfister (\cite{arason_pfister}, also see \cite{lam_quadratic}, Chapter X , Hauptsatz 5.1), the dimension of any positive dimensional anisotropic form in $\mathcal{I}(K)^{l+1}$ is $ \geq 2^{l+1} >2 ^{log_2 2n +1} =4n$. Thus 
$[q_{H_j}] - [q_{H_{j_l}}]  \in \mathcal{I}(L)^{l+1}$ implies that $q_{H_j} \cong q_{H_{j_l}}$, $\forall j \in J_l$.  But as seen before the forms $q_{H_j} $are pairwise inequivalent . Hence we conclude that $|J_l | =1$ and 
\begin{align*}
|I| \leq  |Pic(V)/2 Pic(V)| \cdot  \prod_{i=1}^{l} |H^i (K , \mu_2)^{V}|\cdot |H^i(K, \mu_4^{\otimes i-1})_{V}|
\end{align*}

\section*{acknowledgements}
 The author would like to thank  Andrei Rapinchuk for the many useful discussions with him  while he was at the Institute for Advanced Study, which inspired the research presented in this paper. She  would also like to thank Daniel Krashen for the fruitful conversations on this topic and for his feedback on this work. She is very grateful to Pierre Deligne for suggesting some corrections and making insightful remarks on the results which not only enhanced the quality of this manuscript but also helped her  understand math better. Finally she thanks the anonymous referee for giving  useful suggestions that improved the exposition of this manuscript.  This material is based upon work supported by the National Science Foundation under Grant No. DMS - 1638352.
\nocite*{}
\bibliographystyle{alpha}
\bibliography{ref}

\end{document}